\documentclass[10pt,a4paper]{amsart} 
\usepackage[draft]{optional}
\usepackage[british,american]{babel}

\usepackage{mathrsfs}

\usepackage{amsfonts, amsmath, wasysym}

\usepackage{amssymb,amsthm,
paralist
}

\usepackage{amsfonts,amssymb,amsthm,amsopn}
\usepackage{accents,bbm,bibgerm,dsfont,eucal,esint,paralist,url,verbatim,wasysym}
\usepackage[normalem]{ulem}

\usepackage{
latexsym,
nicefrac,
}

\usepackage[usenames]{color}

\usepackage{url}

\definecolor{darkgreen}{rgb}{0,0.5,0}
\definecolor{darkred}{rgb}{0.7,0,0}
\usepackage[colorlinks, 
citecolor=darkgreen, linkcolor=darkred
]{hyperref}
\usepackage{esint}
\usepackage{bibgerm}
\usepackage[normalem]{ulem}

\theoremstyle{plain}
\newtheorem{lemma}{Lemma}[section]
\newtheorem{thm}[lemma]{Theorem}
\newtheorem{prop}[lemma]{Proposition}
\newtheorem{cor}[lemma]{Corollary}
\theoremstyle{definition}

\newtheorem{rmk}[lemma]{Remark}

\numberwithin{equation}{section}

\newcommand{\al}{\alpha}

\newcommand{\Ga}{\Gamma}
\newcommand{\de}{\delta}

\newcommand{\Om}{\Omega}
\newcommand{\ka}{\kappa}

\newcommand{\si}{\sigma}

\newcommand{\Si}{\Sigma}
\renewcommand{\th}{\theta}

\newcommand{\R}{\ensuremath{{\mathbb R}}}

\newcommand{\Z}{\ensuremath{{\mathbb Z}}}

\newcommand{\upto}{\uparrow}

\newcommand{\norm}[1]{\Vert#1\Vert}  

\newcommand{\beq}{\begin{equation}}
\newcommand{\eeq}{\end{equation}}
\newcommand{\beqs}{\begin{equation*}}
\newcommand{\eeqs}{\end{equation*}}
\newcommand{\beqa}{\begin{equation}\begin{aligned}}
\newcommand{\eeqa}{\end{aligned}\end{equation}}
\newcommand{\beqas}{\begin{equation*}\begin{aligned}}
\newcommand{\eeqas}{\end{aligned}\end{equation*}}
\newcommand{\brmk}{\begin{rmk}}
\newcommand{\ermk}{\end{rmk}}
\newcommand{\partref}[1]{\hbox{(\csname @roman\endcsname{\ref{#1}})}}
\newcommand{\half}{\frac{1}{2}}
\newcommand{\thalf}{\tfrac{1}{2}}

\renewcommand{\i}{\text{i}}
\newcommand{\pt}{\partial_t}
\newcommand{\M}{\ensuremath{{\mathcal M}}_{-1}}

\newcommand{\abs}[1]{\vert#1\vert} 

\newcommand{\eps}{\varepsilon}

\newcommand{\Col}{\mathcal{C}}

\newcommand{\Rea}{\text{Re}}

\newcommand{\kg}{k_g}
\newcommand{\kgc}{k_{g_c}}
\newcommand{\kgn}{k_{g_0}}
\newcommand{\kgt}{k_{\tilde g}}
\newcommand{\bM}{{\partial M}}
\newcommand{\nanu}{\frac{\partial u}{\partial n}}

\newcommand{\nanvg}[1]{\frac{\partial v}{\partial n_{#1}}}
\newcommand{\nanuv}[1]{\frac{\partial {#1}}{\partial n_g}}

\newcommand{\Den}{\Delta_{g_0}}
\newcommand{\dvn}{dv_{g_0}}
\newcommand{\dSn}{dS_{g_0}}
\newcommand{\Area}{\text{Area}}
\newcommand{\tr}{\text{tr}}
\newcommand{\dvg}{dv_g}
\newcommand{\dvgc}[1]{dv_{g_{#1}}}
\newcommand{\dSg}{dS_g}
\newcommand{\dSgc}[1]{dS_{g_{#1}}}
\newcommand{\cc}{\mathtt{c}} 

\newcommand{\Xdb}{\overline{X}_d}

\title{{\sc Hyperbolic metrics on surfaces with boundary}
}
\author{Melanie Rupflin}
\date{\today}

\begin{document}
\begin{abstract}
We discuss an alternative approach to the uniformisation problem on surfaces with boundary by representing conformal structures on surfaces $M$ of general type by hyperbolic metrics with boundary curves of constant positive geodesic curvature.
In contrast to existing approaches to this problem, the boundary curves of our surfaces $(M,g)$ cannot collapse as the conformal structure degenerates which is important in applications in which $(M,g)$ serves as domain of a PDE with  boundary conditions.

\end{abstract}

\maketitle

\section{Introduction}
Given a surface $M$ there are many interesting questions with regards to representing a given conformal structure by a Riemannian metric. 

A classical question in this context, for $M=S^2$ known as Nirenberg's problem, asks what functions can occur as Gauss-curvatures of such metrics on closed surfaces, and over the past decades this problem has been studied by many different authors, we refer in particular to  \cite{Kazdan-Warner, Berger, Chang-Yang,Struwe-flow,Schoen-Yau} as well as the more recent work of \cite{Chen,Borer-G-Struwe} and the references therein for an overview of existing results. We also note that the corresponding problem on surfaces with boundary was investigated in \cite{Cherrier} .

Another classical problem in this context, but of a quite different flavour, is to ask how to `best' represent a given conformal structure by a Riemannian metric. For closed surfaces this problem is  addressed by the classical uniformisation theorem that allows us to represent every conformal structure by a (unique for genus at least 2) metric of constant Gauss-curvature $K_g\equiv 1,0,-1$, while for complete surfaces this problem was addressed by Mazzeo and Taylor in \cite{Mazzeo-Taylor}.
On surfaces with boundary, Osgood, Philips and Sarnak
introduced in  \cite{Sarnak}  two different notions of uniformisation, with uniform metrics of type I characterised by having constant Gauss-curvature and geodesic boundary curves, while uniform metric of type II  are flat and have boundary curves of constant geodesic curvature.
The corresponding heat flows were analysed by Brendle in \cite{Brendle-1}, who proved that these flows admit global solutions which converge to the corresponding uniform metric in the given conformal class. As observed by Brendle in \cite{Brendle-2}, for the two different types of uniform metrics introduced in \cite{Sarnak}, only one of the terms 
on the left hand side of the Gauss-Bonnet formula 
$$\int_M K dv_g+\int_{\bM} k_g dS_g=2\pi\chi(M)$$
gives a contribution and so the two types of uniform metrics can 
be seen the opposite ends of a whole family of metrics for which all terms in the above formula have the same sign. Brendle \cite{Brendle-2} proved also in this more general setting that solutions of the corresponding heat flows exist for all times and converge, now to metrics with $K_g\equiv \bar K$ and $k_g\equiv \bar k$, where the signs of $\bar K$ and $\bar k$ both agree with the sign of $\chi(M)$. We note that the same restriction on the signs of the curvatures is also present in the work of Cherrier \cite{Cherrier}. 

Here we propose an alternative way of representing conformal structures on surfaces of general type with boundary which is motivated by applications to geometric flows, such as Teichm\"uller harmonic map flow \cite{RT, R-cyl} or Ricci-harmonic map flow \cite{Reto}, in which the surface $(M,g)$ plays the role of a time dependent domain on which a further PDE is solved. For this purpose the described ways of uniformisation on surfaces with boundary suffer the serious drawback that a degeneration of the conformal structure, which can occur even for curves of metrics with finite length, can lead to a degeneration of the metric near the boundary curves, with boundary curves turning into punctures in the limit, so that the very set on which the boundary condition is imposed can be lost. 

To resolve this problem, we propose to represent conformal classes on surfaces of general type instead by hyperbolic metrics for which each boundary curve is a curve of \textit{positive} constant geodesic curvature, chosen so that each of the boundary curves gives a fixed positive contribution to the Gauss-Bonnet formula. As we shall see below, this alternative approach has the advantage that the resulting metrics will remain well controlled near the boundary even if the conformal class degenerates in a way that would cause the  boundary curves of the corresponding uniform metrics of type I or II to collapse. The existence of a unique representative of each conformal class with these desired properties is ensured by our first main result.
\begin{thm}\label{thm:1}
Let $M$ be a compact oriented surface of genus $\gamma$ with boundary curves $\Gamma_1,..,\Gamma_k$ and negative Euler characteristic 
and let $d>0$ be any fixed number. Then for any conformal structure $\mathtt{c}$ on $M$ there exists a unique hyperbolic metric $g$ compatible with $\mathtt{c}$ for which 
\beq \label{def:Md} 
\kg\vert_{\Gamma_i}\cdot L_g(\Gamma_i)\equiv d \text{ on } \Gamma_i \text{ for every } i=1\ldots k.
\eeq
Denoting by $\M^d$ the set of all hyperbolic metrics on $M$ satisfying \eqref{def:Md}, we furthermore have that for every $g\in \M^d$ and every $i\in\{1,\ldots,k\}$ there exists a unique simple closed geodesic $\gamma_i$ in the interior of $M$ which is homotopic to $\Gamma_i$, this geodesic is surrounded by a collar neighbourhood $\Col(\gamma_i)$ that is described in Lemma \ref{lemma:col-Md}
 and its length is related to the length of the corresponding boundary curve by \beq 
\label{claim:rel-length}
L_g(\Gamma_i)^2-L_g(\gamma_i)^2=d^2.\eeq
\end{thm}
For hyperbolic surfaces with boundary curves of constant geodesic curvature $k_g$, the relation \eqref{claim:rel-length} between the lengths of a boundary curve and the corresponding geodesic is equivalent to \eqref{def:Md} and we note that \eqref{def:Md} implies that the area of the region enclosed by $\Gamma_i$ and $\gamma_i$ is always equal to $d$. 
We remark that the quantity \eqref{claim:rel-length} appears also naturally if one studies horizontal curves of metrics on closed surfaces, that is curves of hyperbolic metrics which move $L^2$-orthogonally to the action of diffeomorphisms, and that for such curves  the analogue of \eqref{claim:rel-length} is valid  at the infitesimal level, see Remark \ref{rmk:horizontal} for details.

While the uniform metrics of type I could be viewed as the extremal case $d=0$ of $\M^d$, the resulting compactifications of the moduli space are very different. For our class of metrics, the analogue of 
the Deligne-Mumford compactness theorem takes the following form.

\begin{thm}\label{thm:DM}
Let $M$ be a compact oriented surface of genus $\gamma$ with boundary curves $\Gamma_1,..,\Gamma_k$ and negative Euler characteristic, let $\M^d$, $d>0$, be the set of all hyperbolic metrics on $M$ for which \eqref{def:Md} is satisfied and suppose that $g^{(j)}$ is a sequence in $\M^d$ for which the lengths of the boundary curves are bounded above uniformly.
\newline
Then, after passing to a subsequence,  
$(M,g^{(j)})$ converges to a complete hyperbolic surface $(\Sigma,g_\infty)$
with the same number of boundary curves, all satisfying \eqref{def:Md}, where $\Si$ is obtained
from $M$ by removing a collection 
$\mathscr{E}=\{\sigma^{j}, j=1,\ldots,\kappa\}$ of 
$\kappa\in\{0,\ldots,3(\gamma-1)+2k\}$ pairwise disjoint homotopically nontrivial simple closed curves in the interior of $M$ and the convergence is to be understood as follows:
\newline
For each $j$ there exists a collection $\mathscr{E}^{(j)}=\{\si_{i}^{(j)}, i=1,\ldots,\kappa\}$ of
pairwise disjoint simple closed geodesics in $(M,g^{(j)})$
 of length
$L_{g^{(j)}}(\si^{(j)}_{i}) \rightarrow 0\text{ as }j \rightarrow \infty$
and a diffeomorphism $f_j:\Sigma \rightarrow M\setminus \cup_{i=1}^\kappa\si^{(j)}_{i} $ 
such that 
$$f_j^*g^{(j)} \rightarrow g^\infty \text{ smoothly locally on }\Sigma.$$
\end{thm}

The above results assure that the metrics are well controlled near the boundary even if the conformal structure degenerates, compare also Remark \ref{rmk:L-to-infty}, and that in particular no boundary curve can be `lost'. 
Both of these properties are crucial in applications where $(M,g)$ plays the role of the domain of a PDE with prescribed boundary conditions, even more so if we are dealing with Plateau boundary conditions as in the study of Teichm\"uller harmonic map flow, introduced in the joint work with Topping \cite{RT} for closed surfaces, see also \cite{Ding-Li-Liu} for an equivalent flow from tori, and in \cite{R-cyl} for cylinders, where one expects less regularity at the boundary than for comparable Dirichlet boundary conditions.   
In particular, if one hopes to prove global existence results, as obtained in \cite{RT-neg} and \cite{RT-global} for Teichm\"uller harmonic map flow, or in \cite{Reto-Melanie} for Ricci-harmonic map flow for closed surfaces, for surfaces with boundary, it is important that the most delicate region for the PDE, i.e.~the boundary region, and the most delicate region for the evolution of the domain metric, which for hyperbolic metrics are sets of small  injectivity radius, do not overlap but are instead far apart as is the case  for our class of metrics $\M^d$.

This paper is organised as follows. In Section \ref{sect:2} we
consider the problem of finding hyperbolic metrics in a given conformal class with prescribed positive geodesic curvatures $k_g\vert _{\Gamma_i}=c_i$ and analyse the properties of such metrics. The main difficulty here lies in the fact that for $c_i>0$ the boundary condition has the wrong sign to apply known existence results as found e.g.~in \cite{Cherrier} and the corresponding variational problems contain negative boundary terms that have to be analysed carefully.
 Based on the results and estimates proven in Section \ref{sect:2} we will then give the proofs of the main results in Section \ref{sect:3}.

\section{Hyperbolic surfaces with boundary curves of prescribed positive curvature}\label{sect:2}
In this section we prove the existence and uniqueness of hyperbolic metrics for which the boundary curves have prescribed positive constant geodesic curvature and establish several key properties of these metrics, which will be the basis of the proofs of our main results given in Section \ref{sect:3}. We show in particular.

\begin{lemma}
\label{lemma:existence}
Let $M$ be an oriented  surface with boundary $\partial M=\bigcup_{i=1}^k \Gamma_i$ with $\chi(M)=2(1-\gamma)-k<0$ and let 
$\cc$ be any conformal structure on $M$.
Then for any 
$c=(c_1,..,c_k)\in [0,1)^k$
there exists a unique metric 
$g_c$ on $M$ compatible with $\cc$ so that
\beq \label{eq:c}
\left \{ \begin{array}{rll}
K_{g_c}&=-1 &\text{ in } M\\
\kgc &=c_i  & \text{ on } \Gamma_i, \quad i=1,..,k.
\end{array}
\right. 
\eeq
\end{lemma}

This result is of course true also for $c_i<0$, and in that case is indeed easier to prove as the boundary term in the corresponding
variational integral has the right sign. We are however not interested in the properties of representatives with $c_i\leq 0$ as their boundary curves can collapse if the conformal structure degenerates, the very feature of the existing approaches of uniformisation that we want to avoid with our construction.

We recall that under a conformal change $g=e^{2u}g_0$ 
 the Gauss-curvature transforms by
\beq\label{eq:trafo-Gauss}
K_g=e^{-2u}(K_{g_0}-\Delta_{g_0}u)
\eeq
while, denoting by $n_{g_0}$ the outer unit normal of $(M,g_0)$, the geodesic curvature $k_g$ is characterised by
\beq \label{eq:trafo-geod}
\frac{\partial u}{\partial n_{g_0}}+\kgn=\kg\cdot e^{u}. 
\eeq
In the following we let $g_0$ be the unique metric so that $(M,g
_0)$ is hyperbolic with geodesic boundary curves (which can 
e.g.~be obtained by doubling the surface and applying the classical uniformisation theorem), and write for short $n=n_{g_0}$. 
Thus $g=e^{2u}g_0$ satisfies \eqref{eq:c} if and only if 
\beq \label{eq:PDE}
\left \{ \begin{array}{rll}
-\Den u&=1-e^{2u} &\text{ in } M\\
\nanu &=c_i e^u & \text{ on } \Gamma_i, \quad i=1,..,k.
\end{array}
\right. 
\eeq
Lemma \ref{lemma:existence} is hence an immediate consequence of the following more refined result on solutions of the above PDE that we will prove in the present section.

\begin{prop} \label{prop:existence-PDE}
Let $(M,g_0)$ be an oriented hyperbolic surface with geodesic boundary curves $\Gamma_1,\ldots, \Gamma_k$. 
Then for any $c=(c_1,\ldots,c_k)\in [0,1)^k$ the equation \eqref{eq:PDE} has a unique weak solution $u_c\in H^1(M,g_0)$ and this solution is
smooth up to the boundary of $M$. In addition, the map 
$$[0,1)^k\ni c \mapsto u_c\in H^1(M,g_0)$$ is of class $C^1$.
\end{prop} 

We begin by establishing the existence of solutions to \eqref{eq:PDE} based on the direct method of calculus of variations. Solutions of \eqref{eq:PDE} correspond to 
critical points of  
\beq 
\label{eq:Int}
I_c(u)=\int_M\abs{d u}_{g_0}^2 +e^{2u}-2u \dvn -\sum_{i} 2c_i\int_{\Ga_i} e^u \dSn,\eeq
which is well defined on $H^1(M,g_0)$ as the Moser-Trudinger inequality \cite{Moser-Trudinger} and its trace-versions, see e.g.~\cite{Li-Liu-MT}, ensure in particular that for any $q<\infty$ 
\beq \label{est:MT} 
\sup_{u\in H^1(M,g_0), \norm{u}_{H^1(M,g_0)}\leq 1}\int_M e^{q \abs{u}}dv_{g_0}+\int_{\partial M} e^{q \abs{u}}  \dSn<\infty.\eeq
A well known consequence of this estimate is that for every $1<p<\infty$ 
the maps 
\beq
\label{eq:maps-MT}
H^1(M,g_0)\ni u\mapsto e^u\in L^p(M,g_0) \text{ and } H^1(M,g_0)\ni u\mapsto \tr_{\bM}(e^u)\in L^p(\bM,g_0) \eeq
are compact operators: Any bounded sequence in $H^1$ has  a subsequence which converges weakly in $H^1$, strongly in $L^2$ and whose traces converge stongly in $L^2$. The corresponding  sequences $e^{u_n}$ and $\tr_{\bM}(e^{u_n})$ hence converge in measure 
and, thanks to \eqref{est:MT} (applied e.g.~for $q=2p$), are $p$-equiintegrable so converge strongly in $L^p$ by Vitali's convergence theorem. 

An immediate consequence of the compactness of the operators in \eqref{eq:maps-MT} is that $I_c$ is weakly lower semicontinuous on $H^1(M,g_0)$. Hence, to establish the existence of a minimiser of $I_c$ in $H^1(M,g_0)$, and thus a solution of \eqref{eq:PDE}, it suffices to prove that 
$I_c$ is also coercive on $H^1(M,g_0)$. To deal with the negative boundary terms we will use that 
on hyperbolic surfaces with geodesic boundary curves the trace-theorem is valid in the following form, in particular with  leading order term on the right hand side appearing with a factor of $1$.
\begin{lemma}\label{lemma:trace}
For any $\bar L<\infty$ there exists a constant $C_1=C_1(\bar L)<\infty$ so that the estimate 
\beq
\label{est:trace}
\int_{\bM} \abs{w} \dSn \leq \int_M\abs{d w}_{g_0} \dvn +C_1 \int_M \abs{w} \dvn
\eeq
holds true for any oriented hyperbolic surface $(M,g_0)$ with geodesic boundary curves of length $L_{g_0}(\Gamma_i)\leq \bar L$, $i=1,\ldots, k$, and every $w\in W^{1,1}(M,g_0)$.
\end{lemma}

\begin{proof}[Proof of Lemma \ref{lemma:trace}] 
We derive this estimate from the corresponding trace-estimate 
\beq 
\label{est:eucl-cyl}
\int_{\{0\}\times S^1} \abs{w} d\th\leq \int_{0}^{X}\int_{S^1} \abs{\partial_s w} d\th ds+X^{-1} \int_{0}^{X}\int_{S^1} \abs{w} d\th ds
\eeq 
on euclidean cylinders $[0,X]\times S^1$ and the properties of hyperbolic collars as follows. We first recall that the classical Collar lemma of Keen-Randol \cite{randol} yields the existence of pairwise disjoint neighbourhoods 
$\Col(\Gamma_i)$ of the boundary curves which are isometric to the cylinders 
$(-X(\ell_i),0]\times S^1$ equipped with  
$\rho_{\ell_i}(s)^2 (ds^2+d\theta^2)$, where
$\ell_i=L_{g_0}(\Gamma_i) $ and where 
\beq \label{def:rho-X}
\rho_{\ell}(s)=\tfrac{\ell}{2\pi} (\cos(\tfrac{\ell}{2\pi}s))^{-1} \text{ and 
} 
X(\ell)=\tfrac{2\pi}{\ell}\big(\tfrac\pi2-\arctan(\sinh(\tfrac{\ell}{2})) \big),\eeq
with the 
boundary curve $\Ga_i$ corresponding to $\{0\}\times S^1$. 
We hence obtain from \eqref{est:eucl-cyl} that
\beqas 
\int_{\Gamma_i}\abs{w} \dSn &=
\rho_{\ell_i}(0)\int _{\{0\}\times S^1} \abs{w} d\th
\leq \int_{0}^{X(\ell_i)}\int_{S^1} \rho_{\ell_i}^{-1}\abs{\partial_s w} \rho_{\ell_i}^2 d\th ds+\frac{1}{X(\ell_i)\rho_{\ell_i}(0)} \int_{0}^{X(\ell_i)}\int_{S^1} \abs{w} \rho_{\ell_i}^2d\th ds\\
&\leq \int_{\Col(\Gamma_i)} \abs{d w}_{g_0} \dvn+c_{\bar L}^{-1} \int_{\Col(\Gamma_i)} \abs{w} \dvn,
\eeqas
 for every $i$, where we use that 
$\rho_\ell(s)\geq \rho_\ell(0) \geq \frac{c_{\bar L}}{X(\ell)}$ for some  $c_{\bar L}>0$ and  $\ell\in (0,\bar L]$.
As  the collars are disjoint this implies the claim of the lemma. 
\end{proof}

Returning to the proof of the first part of Proposition \ref{prop:existence-PDE}, and hence of the coercivity of $I_c$ defined in \eqref{eq:Int}, we now set
  $\bar c:=\max\{c_i\}<1$ and apply Lemma \ref{lemma:trace} to  bound 
\beqas
\sum_i c_i\int_{\Gamma_i}e^u\dSn  &\leq \bar c\int_{\bM}e^u \dSn 
\leq  \frac12 \bar c\int\abs{d u}_{g_0}^2+e^{2u} \dvn +C_1\int e^u \dvn,
\eeqas
where all integrals are computed over $M$ unless specified otherwise. 
Writing $-2u=2\abs{u}-4u^+$ for $u^+=\max\{u,0\}$, we can thus estimate 
\beqas 
I_c(u) &\geq  (1-\bar c)\int \abs{d u}_{g_0}^2+e^{2u} \dvn + 2\int \abs{u} \dvn -4 \int u^+ \dvn -2C_1\int e^u \dvn\\
& \geq  \thalf (1-\bar c)\int \abs{d u}_{g_0}^2+e^{2u} \dvn +2\int \abs{u} \dvn 
+\int\tfrac{1}{2}(1-\bar c)e^{2u^+}-4u^+-2C_1e^{u^+} \dvn\\
&\geq  \thalf(1-\bar c)\int \abs{d u}_{g_0}^2+e^{2u} \dvn+2\int \abs{u} \dvn -C
\eeqas
for a constant $C$ that is allowed to depend on $\bar c\in [0,1)$, $\chi(M)$, and hence $\Area(M,g)=-2\pi\chi(M)$, and an upper bound $\bar L$ on the length of the boundary curves of $(M,g_0)$.

Coercivity of $I_c$ now easily follows: 
If  $\norm{d u_n}_{L^2(M,g_0)}\to \infty$ then clearly 
$I_c(u_n)\to \infty$ while for sequences with 
$\norm{u_n}_{H^1(M,g_0)}\to \infty$ and $\norm{d u_n}_{L^2(M,g_0)}\leq C$, the Poincar\'e inequality implies that also
$\abs{\fint u_n \dvn}\to \infty$, so 
$I_c(u_n)\geq 2\int \abs{u_n} \dvn -C \geq 2\Area(M,g_0)\cdot \abs{\fint  u_n \dvn} -C\to \infty.$

This establishes the existence of a weak solution $u_c\in H^1(M,g_0)$ to \eqref{eq:PDE} for any $c\in[0,1)^k$. 
Since the non-linearity in the Neumann-problem \eqref{eq:PDE} is subcritical, the 
regularity theorem \cite[Th\'eor\`eme 1]{Cherrier} of Cherrier applies and yields that every weak solution of \eqref{eq:PDE} is indeed smooth up to the boundary. At the same time we remark that we could not have used the results of \cite{Cherrier} to establish existence of solutions, as our boundary data has the wrong sign.

\begin{rmk}\label{rmk:cylinder-Z}
As we only use that the geodesic curvature $k_g$ is strictly less than $1$, the above proof indeed shows that for any given functions $k_i\in L^p(\Gamma_i)$, $p>1$, for which $k_i\leq \bar c$ for some $\bar c<1$, there exists a hyperbolic metric $g$ compatible to $\cc$ with $k_g=k_i$ on $\Gamma_i$, $i=1,\ldots, k$. 
\end{rmk}
We will prove the other claims of Proposition \ref{prop:existence-PDE} at the end of the section based on properties of the surfaces $(M,g_c)$ that we discuss now, including the following version of the collar lemma.

\begin{lemma} \label{lemma:collar}
Let $(M,g)$ be an oriented hyperbolic surface with boundary curves of constant geodesic curvature $\kg\vert_{\Gamma_i}\equiv c_i\in [0,1)$.  Then for each $i\in\{1,\ldots, k\}$ there exists a unique simple closed geodesic $\gamma_i$ in $(M,g)$ that is homotopic to $\Gamma_i$ and there exist pairwise  disjoint neighbourhoods $\Col(\Gamma_i)$ of the boundary curves $\Gamma_i$ in $(M,g)$ which are isometric to cylinders
$$(-X(\ell_i),Y(\ell_i,c_i)]\times S^1 \text{ with metric } \rho_{\ell_i}(s)^2(ds^2+d\th^2)$$
where $\rho_{\ell_i}$ and $X(\ell_i)$ are given by \eqref{def:rho-X},  $\ell_i=L_g(\gamma_i)$, and where 
\beq 
\label{def:Y}
Y(\ell_i,c_i)=\tfrac{2\pi}{\ell_i} \arcsin(c_i).
\eeq
In these coordinates $\Gamma_i$ corresponds to $\{Y(\ell_i,c_i)\}\times S^1$ while $\gamma_i$ corresponds to $\{0\}\times S^1$.
\end{lemma}

\begin{proof}
We note that since our surface is hyperbolic, 
 the Dirichlet energy of maps $u:S^1\to (M,g)$ has a unique minimiser in the homotopy class of $\Gamma_i$, c.f.~\cite{Eells-Sampson}, which  coincides with $\Gamma_i$ if $c_i=0$. Otherwise, $\Gamma_i$ has positive geodesic curvature so the image of this minimiser must lie in the interior of $M$ and hence be the desired simple closed geodesic. 

We let $\Col^+(\Gamma_i)$  be the connected component of $\M\setminus \bigcup_i \gamma_i$ that is bounded by $\gamma_i$ and 
$\Gamma_i$ and
set $M_0:= M\setminus \bigcup_i \Col^+(\Gamma_i)$. 
As $(M_0,g)$ is hyperbolic with geodesic boundary, the Collar lemma \cite{randol} gives disjoint neighbourhoods $\Col^-(\gamma_i)$ of $\gamma_i$ in $M_0$ that are 
isometric to $\big((-X(\ell_i),0]\times S^1,\rho_{\ell_i}(s) (ds^2+d\th^2)\big)$, with $\rho_\ell$ and $X(\ell)$ given by \eqref{def:rho-X}. 
The resulting disjoint neighbourhoods $\Col(\Gamma_i):= \Col^-(\gamma_i)\cup \Col^+(\Gamma_i)$ of $\Gamma_i$ in our original surface $M$ are bounded by curves of constant geodesic curvature and are isometric to a subset of the complete hyperbolic cylinder $ \big((-\tfrac{\pi^2}{\ell_i}, \tfrac{\pi^2}{\ell_i})\times S^1, \rho_{\ell_i}^2(ds^2+d\th^2)\big)$ around a geodesic of length $\ell_i$, where such an isometry can e.g.~be obtained by using the fibration of $\Col(\Gamma_i)$ by the geodesics that cross $\gamma_i$ orthogonally. We note that the only closed curves of constant geodesic curvature in such a cylinder are circles $\{s\}\times S^1$, whose curvature is
\beq k_g=\rho^{-1}\frac{\partial}{\partial s} \log(\rho(s))+k_{g_{eucl}}=\rho^{-2}\partial_s \rho =\sin(\tfrac{\ell}{2\pi}s),
\label{eq:curv-circles} \eeq
compare \eqref{eq:trafo-geod}; indeed, comparing the curvature of any other closed curve $\si$ 
with the one of the circles 
$\{s_{\pm}\}\times S^1$ through points $P_\pm=(s_\pm,\th_\pm)$ of $\si$ with extremal $s$ coordinate we get
$$\kg(\sigma)(P_+)\geq \kg(\{s_+\}\times S^1) =\sin(\tfrac{\ell}{2\pi}s_+)>\sin(\tfrac{\ell}{2\pi}s_-)
\geq  \kg(\sigma)(P_-).$$
The collar neighbourhood $\Col(\Gamma_i)$ obtained above must hence be isometric to a cylinder $((-X,Y]\times S^1,\rho_{\ell_i}^2(ds^2+d\th^2))$ where, by \eqref{eq:curv-circles}, $X$ and $Y$ are as described in the lemma. 
\end{proof}

We also use the following standard property of Riemann surfaces.  
\begin{rmk} \label{rmk:cyl}
For any given oriented Riemann surface $(M,\mathtt{c})$ with boundary curves $\Gamma_1,\ldots \Gamma_k$ 
there exists a number $\bar Z$ so that the following holds true.
Let $U$ be any neighbourhood of one of the boundary curves $\Gamma_i$ which is conformal to a cylinder $(0,Z]\times S^1$. Then 
$Z\leq \bar Z.$
\end{rmk}
We include a short proof of this remark in the appendix and combine it with Lemma \ref{lemma:collar} to get
\begin{cor}
\label{cor:lower-bound-ell}
For any  conformal structure $\mathtt{c}$ on $M$ there exists a $\de>0$ so that the following holds true. Let  $g$ be any hyperbolic metric on $M$ for which $\kg\vert_{\Gamma_i}\equiv c_i\in [0,1)$, $i=1\ldots k$, and let $\gamma_i$ be the geodesics in $(M,g)$ that are homotopic to the boundary curves $\Gamma_i$. Then 
\beq
\label{est:lower-bound-ell}
\ell_i:=L_{g_c}(\gamma_i)\geq \delta
\text{ and }
L_{g_c}(\Gamma_i)=\frac{\ell_i}{\sqrt{1-c_i^2}} \geq \frac{\de}{\sqrt{1-c_i^2}},
\eeq
in particular $L_{g_c}(\Gamma_i)\to \infty$ as $c_i\upto 1$.
\end{cor}

The bound on $\ell_i$ follows directly from Lemma \ref{lemma:collar} and Remark \ref{rmk:cyl}, applied for $Z= X(\ell_i)\to \infty$ as $\ell_i\to 0$, while the expression for $L_{g_c}(\Gamma_i)$ follows from  \eqref{def:rho-X} and
 \eqref{def:Y}.

For these surfaces we can now prove the following version 
of the trace-theorem.
\begin{lemma}\label{lemma:trace2}
Let $(M,g)$ be an oriented hyperbolic surface with boundary curves of constant geodesic curvature $\kg\vert_{\Gamma_i}\equiv c_i\in [0,1)$ and let
 $\Col^+(\Gamma_i)$ be the 
subset of the collar $\Col(\Gamma_i)$ described in Lemma \ref{lemma:collar} that is bounded by $\Gamma_i$ and the corresponding geodesic $\gamma_i$.
Then 
\beq
\label{est:trace2}
c_i\int_{\Gamma_i} \abs{w} \dSg \leq \int_{\Col^+(\Gamma_i)}\abs{w}  \dvg+c_i\int_{\Col^+(\Gamma_i)}\abs{dw}_g \dvg
\eeq
holds true for any $w\in W^{1,1}(M,g)$.
Furthermore,
there exists $\eps>0$, allowed to depend on both the lengths $\ell_i$ of the geodesics $\gamma_i$ and the curvatures $c_i$, so that for every $w\in W^{1,1}(M,g)$
\beq \label{est:trace3}
\int_{\bM} (\kg+\eps) \abs{w} \dSg\leq (1-\eps) \int_M \abs{w}\dvg+(1-\eps)\int_M \abs{dw}_g \dvg.
\eeq
\end{lemma}
We note that the above lemma assures in particular that if  $w\in H^1(M,g)$, then 
\beq \label{est:trace-L2}
\int_{\bM} (\kg+\eps) w^2 \dSg\leq (1-\eps) \int_M \abs{dw}_g^2+2w^2\dvg.
\eeq
\begin{proof}[Proof of Lemma \ref{lemma:trace2}]
From Gauss-Bonnet and \eqref{eq:curv-circles} we obtain that
 for $s\in[0,Y_i]$, $Y_i=Y(\ell_i,c_i)$
\beqs 
\int_0^s\rho_{\ell_i}^2(x) dx=\tfrac{1}{2\pi}\Area_{g}([0,s]\times S^1)= 
\tfrac{1}{2\pi}\int_{\{s\}\times S^1} \kg dS_g=\rho_{\ell_i}(s)\kg\vert_{{\{s\}\times S^1}}=\rho_{\ell_i}(s)
\sin(\tfrac{\ell_i}{2\pi}s),
\eeqs
$(s,\th)$ collar coordinates on $\Col(\Gamma_i)$, 
in particular $\int_0^{Y_i}\rho_{\ell_i}^2=\rho_{\ell_i
}(Y_i)\cdot c_i$. Multiplying 
\beq 
\label{eq:proof-trace2}
\int_{\{Y_i\}\times S^1}\abs{w} d\theta=\int_{\{s_0\}\times S^1} \abs{w} d\theta+\int_{s_0}^{Y_i}\int_{S^1}\partial_s\abs{w} d\theta ds,
\eeq
with $\rho_{\ell_i}(s_0)^2$ and integrating over $s_0\in[0,Y_i]$ using Fubini 
hence gives the desired bound of
\beqa 
\label{est:trace-proof-trace}
c_i\int_{\Gamma_i}\abs{w} \dSg=&\int_{\Col^+(\Gamma_i)}\abs{w} \dvg +\int_0^{Y_i}\int_{S^1} \partial_s \abs{w}\cdot  \rho_{\ell_i}(s)\sin(\tfrac{\ell_i}{2\pi}s) d\theta ds\\
\leq & \int_{\Col^+(\Gamma_i)}\abs{w} \dvg +c_i\cdot 
\int_{\Col^+(\Gamma_i)}  \abs{dw}_g \dvg.
\eeqa
Multiplying \eqref{eq:proof-trace2} with $\rho_{\ell_i}(Y_i)$ and averaging over $s_0\in (-X_i,0]$, $X_i=X(\ell_i)$, also yields
\beqa \label{est:trace-proof-cor}
\int_{\Gamma_i}\abs{w} dS_g&\leq \rho_{\ell_i}(Y_i)\cdot X_i^{-1}
\int_{-X_i}^0\int_{S^1} \abs{w} d\theta ds+\rho_{\ell_i}(Y_i)\int_{-X_i}^{Y_i}\int_{S^1}\abs{\partial_s w} d\theta ds\\
&\leq C_2
\int_{\Col(\Gamma_i)\setminus \Col^+(\Gamma_i)}\abs{w} \dvg +C_3 \int_{\Col(\Gamma_i)}\abs{dw}_g \dvg\eeqa
now for constants $C_{2,3}$ that depend both on $\ell_i=2\pi\min\rho_{\ell_i}(\cdot)$ and $L_{g}(\Gamma_i)=2\pi \rho_{\ell_i}(Y_i)$. 
To obtain the second claim of the lemma, we now combine 
 \eqref{est:trace-proof-trace}, multiplied by $(1-\eps)$ for some $\eps\in (0,1)$ chosen below, and \eqref{est:trace-proof-cor}, multiplied by $(1+c_i)\cdot\eps$, to conclude that 
\beqas
(c_i+\eps)\int_{\Gamma_i}\abs{w} dS_g&\leq  (1-\eps)
\int_{\Col^+(\Gamma_i)}\abs{w} \dvg+C_2(1+c_i)\cdot\eps\int_{\Col(\Gamma_i)\setminus \Col^+(\Gamma_i)}\abs{w} \dvg \\
&\quad +[c_i(1-\eps)
+C_3(1+c_i)\cdot \eps] \int_{\Col(\Gamma_i)}\abs{dw}_g \dvg.\eeqas
For $\eps>0$ chosen small enough to ensure that 
$2C_2\eps\leq 1-\eps$ and $2C_3\eps\leq (1-c_i)(1-\eps)$, this yields the second claim \eqref{est:trace3} of the lemma as the collar neighbourhoods are disjoint. 
\end{proof}
We are now in a position to prove the following a priori bounds
for PDEs related to \eqref{eq:PDE}
\begin{lemma}\label{lemma:apriori}
Let $M$ be an oriented  surface with boundary curves $\Gamma_1,\ldots ,\Gamma_k$ and let $g$ be a metric on $M$ which satisfies \eqref{eq:c} for some $c\in [0,1)^k$. 
\newline
Then there exist constants $C_{4,5}$, allowed to depend both on $c$ and the underlying conformal structure, 
so that the following holds true for any  $f\in L^2(M,g)$ and $h\in L^2(\bM,g)$.
\begin{enumerate}
\item[(i)] 
Suppose that $w\in H^1(M,g)$ is a weak solution of 
\beqa
\label{eq:PDE-f-h}
-\Delta_g w&=1-e^{2w}+f \quad  \text{ in } M \quad \text{ with } \quad 
\nanuv{w} &=k_g (e^w-1)+h \text{ on } \bM \quad 
\eeqa
for which furthermore  $e^{w}\in H^1(M,g)$. Then 
\beq
\label{est:apriori-nl}
\int_M \abs{d u}_{g}^2 e^w+(e^w-1)^2 (e^w+1) \dvg +\int_{\bM} (e^w-1)^2 \dSg \leq C_4(\norm{f}_{L^2(M,g)}^2+\norm{h}_{L^2(\bM, g)}^2).
\eeq
 \item[(ii)] There exists a unique solution $v\in H^1(M,g)$ of the linearised problem 
  \beqa
\label{eq:PDE-f-h-lin}
-\Delta_g v+2v&=f \quad  \text{ in } M\quad\text{ with } \quad  \nanvg{g} &=k_g v+h \text{ on } \bM.
\eeqa
and we have that 
\beq
\label{est:apriori-l}
\norm{v}_{H^1(M,g)}^2 \leq  C_5(\norm{f}_{L^2(M,g)}^2+\norm{h}_{L^2(\bM,g)}^2).
\eeq
\end{enumerate}
\end{lemma}
\begin{proof}[Proof of Lemma \ref{lemma:apriori}]
Let $w$ be as in the first part of the lemma and let $\eps>0$ be as in Lemma \ref{lemma:trace2}. 
Testing 
 \eqref{eq:PDE-f-h} with $e^{w}-1\in H^1(M,g)$, we may estimate
\beqas 
I&:= \int\abs{d w}_{g}^2e^w+(e^w-1)^2(e^w+1) \dvg=\int_{\bM} \frac{\partial w}{\partial n_g}(e^w-1) \dSg+\int f\cdot (e^w-1) \dvg
\\
&=\int_{\bM} k_g (e^w-1)^2 +h(e^w-1)\dSg+\int f \cdot (e^w-1) \dvg\\
&\leq 
\int_{\bM} (\kg+\eps)(e^w-1)^2 \dSg-\tfrac{\eps}2\int_{\bM}(e^w-1)^2 \dSg +\tfrac{\eps}{2}I+
\tfrac{1}{2\eps}[\norm{h}_{L^2(\bM,g)}^2+\norm{f}_{L^2(M,g)}^2].
\eeqas
By \eqref{est:trace3}, the first term on the right is bounded by 
$(1-\eps)\int(1-e^w)^2+2 e^w \abs{dw}_g \abs{1-e^w}dv_g\leq (1-\eps)I$, so the first claim \eqref{est:apriori-nl} of the lemma immediately follows. 

To prove  the second part of the lemma, we 
use that 
the variational integral 
\[F_{f,h}(v):= \int \abs{d v}_g^2+2v^2+2fv\, \dvg-\int_{\bM}\kg v^2+2vh dS_g\]
associated with  \eqref{eq:PDE-f-h-lin} %
 is  coercive, as \eqref{est:trace-L2} implies
\beqa\label{est:lower-F}
F_{f,h}(v)&\geq \eps \int \abs{d v}_g^2+2v^2 dv_g
-2\norm{f}_{L^2(M,g)}\cdot \norm{v}_{L^2(M,g)}-
\norm{h}_{L^2(\bM,g)}\norm{v}_{L^2(\bM,g)}\\
 & \geq\eps \norm{v}_{H^1(M,g)}^2- 
(2\norm{f}_{L^2(M,g)}+C\norm{h}_{L^2(\bM,g)})\norm{v}_{H^1(M,g)}.
\eeqa
Hence  $F_{f,h}$ has a minimiser $v$ which is of course a solution of \eqref{eq:PDE-f-h-lin}, and satisfies $F_{f,h}(v)\leq F_{f,h}(0)=0$ which, combined with \eqref{est:lower-F}, furthemore yields the claimed 
 a priori estimate \eqref{est:apriori-l}.
Finally, this solution of \eqref{eq:PDE-f-h-lin} is unique as the difference $v$ of two solutions of  \eqref{eq:PDE-f-h-lin} 
satisfies 
$$0=\int \abs{d v}_g^2+2v^2 \dvg -\int_\bM \kg v^2 dS_g\geq \eps \norm{v}_{H^1(M,g)}^2,$$
again by  \eqref{est:trace-L2}, and must thus vanish. 
\end{proof}
As a next step towards completing the proof of Proposition \ref{prop:existence-PDE} we show 
\begin{lemma}
\label{lemma:ubc}
Let $M$ be as in Lemma \ref{lemma:apriori} and let $g$ be any metric for which \eqref{eq:c} holds true for some  $c\in[0,1)^k$. Then there exist numbers $0<\eps_0<1-\max{c_i}$ and $C_{6}<\infty$ so that for any $b\in [-\eps_0,\eps_0]^k$ and any hyperbolic metric 
$\tilde g=e^{2w} g$  with 
$\kgt\vert_{\Gamma_i}=c_i+b_i$, $i=1,\ldots, k$,
we have 
\beq 
\label{est:apriori-nl-b}
\int_M (e^w+1)(e^w-1)^2+e^w\abs{d w}_g^2 dv_g + \norm{w}_{H^1(M,g)}^2\leq C_6\max\abs{b_i}^2.
\eeq
In particular, the solution of \eqref{eq:c} is unique.
\end{lemma}

\begin{proof}
Let $b\in [-\eps_0,\eps_0]^k$, where $\eps_0>0$ is determined later, set $\bar b=\max \abs{b_i}$ and suppose that $\tilde g=e^{2w}g$ is as in the lemma. From 
\eqref{eq:trafo-Gauss} and \eqref{eq:trafo-geod}, we obtain that $w$ solves 
\beqa
\label{eq:PDE-w-tilde-g}
-\Delta_g w&=1-e^{2w} \quad  \text{ in } M \quad \text{ with } \quad 
\nanuv{w} &=k_g (e^w-1)+b_i e^w \text{ on } \Gamma_i \quad 
\eeqa
i.e.~satisfies \eqref{eq:PDE-f-h} for 
 $f\equiv 0$ and $h\vert_{\Gamma_i}=b_i e^w$.  
 We note that $e^{w}\in H^1(M,g)$, as we may characterise $w=u_{c+b}-u_c$ as difference of smooth
 solutions of \eqref{eq:c}, so we may bound
 $I:= \int_M \abs{d w}_g^2 e^w+(e^w+1)(e^w-1)^2 \dvg +\int_{\bM}(e^w-1)^2\dSg$  using the first part of Lemma \ref{lemma:apriori} by
\beqas 
I& \leq C_4 \norm{b e^w}_{L^2(\bM,g)}^2 \leq 2C_4 \bar b^2 L_g(\bM)+2C_4 \bar b^2\int_{\bM}(e^w-1)^2\leq C\bar b^2+C\eps_0^2 I.
\eeqas 
For  $\eps_0>0$ sufficiently small this gives the bound 
$I\leq C\bar b^2$ on $I$ claimed in the lemma and it remains 
to establish the analogue bound on the $H^1$ norm of $w$. We  first note that
\beq \label{est:proof-apr-eu}
 \norm{e^{w/2}-1}_{H^1(M)}^2= \int \tfrac14 \abs{d w}_g^2 e^w+(e^{w/2}-1)^2 \dvg
 \leq I \leq  
 C\bar b^2,\eeq
 where norms are computed with respect to $g$ and integrals over $M$ unless indicated otherwise. 
In particular  $\norm{e^{w/2}}_{H^1(M)}\leq C$, and so of course $\norm{e^{3w/2}}_{L^4(M)}+\norm{e^{w}}_{L^2(\partial M)}\leq C$,  where all constants are allowed to depend on $(M,g)$ but not on $b$.
Writing \eqref{eq:PDE-w-tilde-g} in the form 
$$-\Delta_g w=(1-e^{w/2})\cdot (1+e^{w/2}+e^w+e^{3w/2}) \text{ on } M,\qquad \frac{\partial w}{\partial n_g} =c_i(e^{w/2}-1)(e^{w/2}+1)+b_ie^w \text{ on } \Gamma_i,$$
and testing this equation with $w-\bar w_M$, $\bar w_M:= \fint_M wdv_g$, thus allows us to bound 
\beqas
\norm{d w}_{L^2(M)}^2 &\leq 
C\norm{w-\bar w_M}_{L^2(M)}  \norm{1-e^{w/2}}_{L^4(M)}\cdot(1+\norm{e^{3u/2}}_{L^{4}(M)})\\
&\quad +\norm{w-\bar w_M}_{L^2(\bM)}\cdot \big[\norm{e^{w/2}-1}_{L^4(\bM)}\cdot  \norm{e^{w/2}+1}_{L^4(\bM)}+ \bar b \norm{e^{w}}_{L^2(\bM)}\big]\\
&\leq C\norm{w-\bar w_M}_{H^1(M)}\cdot \big[\norm{e^{w/2}-1}_{H^1(M)}+\bar b\big]\leq C \bar{b}\cdot  \norm{d w}_{L^2(M)}.
\eeqas 
Having thus shown that $\norm{dw}_{L^2(M)}\leq C\bar b$, it now remains to 
 show that also $\abs{\bar w_M }\leq C\bar{b}$, which of course follows if we prove that 
 $\int \abs{w}\dvg\leq C\bar{b}$. 
As $\abs{x}\leq 2e^{-\min(x/2,0)}\abs{e^{x/2}-1}$, $x\in\R$, we already obtain from 
\eqref{est:proof-apr-eu} that
$$\int_{\{w>-4\}} \abs{w}\dvg\leq 2e^2 \int_M\abs{e^{w/2}-1} \dvg \leq  C\bar{b},$$
so it remains to bound the corresponding integral over $\{w<-4\}$.
To this end we note that as $\norm{e^{w/2}}_{L^4(M)}\leq C$, we obtain from \eqref{est:proof-apr-eu} that, after reducing $\eps_0$ is necessary, 
\beqas
\int(e^w-1)^2 dv_g &\leq C \norm{e^{w/2}-1}_{H^1(M)}^2 \leq C\bar b^2 \leq C\eps_0^2 \leq \half \Area(M),\eeqas 
so
$\Area_g(\{w<-3\})\leq \Area_g(\{ (e^{w}-1)^2>\tfrac34\})\leq \tfrac43\int (e^w-1)^2dv_g
\leq \tfrac23\Area_g(M).$
Hence $v=(w+3)_{-}=\max(-(w+3),0)$ vanishes on a set of measure at least $\alpha =\tfrac13\Area_g(M)$, so the variant of the Poincar\'e inquality 
$$\norm{v}_{L^2(M)}\leq C\norm{d v}_{L^2(M)}, \qquad C=C(\al, (M,g))$$
valid for such functions $v$ implies that 
also 
\beqas
\int_{\{w<-4\}} \abs{w} \dvg\leq 4\int_{\{w<-4\}} v \dvg
\leq C \norm{d v}_{L^2} \leq C \norm{d w}_{L^2}\leq C\bar{b},
\eeqas
which completes the proof of the lemma. 
\end{proof}

Lemmas \ref{lemma:apriori} and \ref{lemma:ubc} represent the main steps in the proof of the remaining claims of Proposition  \ref{prop:existence-PDE}, which now follow by the following standard argument.

Let $S:c\mapsto  u_c\in H^1(M,g_0)$ be the map that assigns to each $c\in[0,1)^k$  the unique solution $u_c$ of \eqref{eq:PDE}. We claim that $S$ is $C^1$ 
 with $dS(c)(b)= v_{c,b}$,
 for  $v_{c,b}$ 
the unique solutions of 
\beq \label{eq:vbc}
-\Delta_{g_c} v+2v=0 \text{ in } M, \quad \text{ with }\quad \frac{\partial v}{\partial n_{g_c}}=c_iv+b_i \text{ on } \,\Gamma_i, \,i\in\{1,\ldots,k\}. \eeq
Here and in the following  $g_c=e^{2u_c}g_0$ is the unique metric satisfying \eqref{eq:c}.

Given $c\in[0,1)^k$, $b\in \R^k$, say with $\abs{b}=1$, and $\abs{\eps}\leq 1-\max c_i$, we let
$c_\eps=c+\eps b$ and set  
$w_{\eps}:=S(c_\eps)-S(c)$. As $g_{c_\eps}=e^{2S(c_{\eps})}g_0=e^{2w_{\eps}}g_c$ is hyperbolic with $k_{g_{c_\eps}}=c_\eps$, we have
\beq 
\label{eq:u_bc}
-\Delta_{g_c} w_\eps+e^{2w_\eps}-1=0 \text{ in } M \quad \text{ with } \quad \frac{\partial w_\eps}{\partial n_{g_c}}=(c_i+\eps b_i)e^{w_\eps}-c_i \text{ on } {\Gamma_i},\eeq
compare \eqref{eq:trafo-geod}. 
Hence
$\beta_\eps:=S(c_\eps)-S(c)-\eps v_{c,b}= w_\eps-\eps v_{c,b}$ 
solves   
\eqref{eq:PDE-f-h-lin} for $g=g_c$,
$f=1+2w_\eps-e^{2w_\eps}$ and $h\vert_{\Gamma_i}=c_i(e^{w_\eps}-(1+w_\eps))+\eps b_i(e^{w_\eps}-1)$, 
so for functions with
$$\abs{f}\leq
2e^{2(w_\eps)_+} \cdot (w_\eps)^2
\leq 2(1+e^{2w_\eps})\cdot w_\eps^2
\text{ and } 
\abs{h}
\leq \tfrac12 e^{(w_\eps)_+}w_\eps^2+\eps e^{(w_\eps)_+}\abs{w_\eps} 
\leq   (1+e^{w_\eps})\cdot\big[w_\eps^2+\tfrac12\eps^2\big].$$
We recall from 
 Lemma \ref{lemma:ubc} that the $H^1$ norms of $e^{w_\eps}$, $\abs{\eps}\leq \eps_0$, are uniformly bounded, and hence so are $\norm{e^{2w_\eps}}_{L^4(M)}$ and $\norm{e^{w_\eps}}_{L^4(\bM)}$. Using
 \eqref{est:apriori-l} as well as that 
 $w_\eps=\beta_\eps+\eps v_{c,b}$  we thus get
\beqas
\norm{\beta_\eps}_{H^1(M,g_c)} & \leq
C(\norm{f}_{L^2(M,g_c)}+\norm{h}_{L^2(\bM,g_c)})
\leq C \norm{w_\eps^2}_{L^4(M,g_c)}+C \norm{w_\eps^2}_{L^4(\bM,g_c)}+C\eps^2\\
 &\leq C\norm{\beta_\eps}_{H^1(M,g_c)}^2+C\eps^2(1+\norm{v_{c,b}}_{H^1(M,g_c)}^2) \leq  C\norm{\beta_\eps}_{H^1(M,g_c)}^2+C\eps^2,
\eeqas
where we use in the last step that \eqref{est:apriori-l} yields a  bound on the norm of $v_{c,b}$ that is independent of $b$. 
As we know a priori that 
$\norm{\beta_\eps}_{H^1(M,g_c)}\leq \norm{w_\eps}_{H^1(M,g_c)}+\eps \norm{v_{c,b}}_{H^1(M,g_c)} \leq C\eps$, compare Lemma \ref{lemma:ubc}, we thus conclude that  $\norm{\beta_\eps}_{H^1(M,g_c)}\leq C\eps^2$ and thus that 
$S$ is indeed Fr\'echet differentiable in $c$ with $df(c)(b)=v_{c,b}$ as claimed.

We finally remark that $v_{c,b}$ depends continuously on $c$ as can be readily seen by using that $g_{\tilde c}=e^{2(S(\tilde c)-S(c))}$ to view $v_{\tilde c,b}$  as solution of \eqref{eq:PDE-f-h-lin} for $g=g_c$, $f= 2v(1-e^{2(S(\tilde c)-S(c))})$ and $h\vert_{\Gamma_i}= b_i+ (e^{S(\tilde c)-S(c)}-1)(c_i v+b_i)+e^{S(\tilde c)-S(c)} (\tilde c_i-c_i)v$
and applying Lemmas \ref{lemma:apriori} and \ref{lemma:ubc}.

\section{Proof of the main results}
\label{sect:3}

Based on the results of Section \ref{sect:2} we can now show
the first part of Theorem \ref{thm:1}
by proving
\begin{lemma}\label{lemma:diffeo}
Let $(M,\mathtt{c})$ be a compact oriented Riemann surface with boundary curves $\Gamma_1,\ldots,\Gamma_k$ and denote by $g_c$, $c\in [0,1)^k$, the unique metric compatible to $\mathtt{c}$ for which \eqref{eq:c} holds. 
Then the map 
$f:c\mapsto (c_i\cdot L_{g_c}(\Gamma_i))_i$
is a diffeomorphism from $(0,1)^k$ to $
(\R^+)^k$. 
In particular, for every $d>0$ there exists a unique hyperbolic metric $g$ that is compatible with $\mathtt{c}$ and that satisfies \eqref{def:Md}. 
\end{lemma}
\begin{proof}
We first remark that $f:c\mapsto (c_i\cdot L_{g_c}(\Gamma_i))_i$ is a $C^1$ map from $[0,1)^k$ to $(\R^+_0)^k$
as Proposition \ref{prop:existence-PDE} establishes that $c\mapsto u_c$ is $C^1$ into $H^1$,  while the trace version of the Moser-Trudinger inequality implies that $H^1(M,g_0) \ni u\mapsto \int_{\Gamma_i} e^u \dSn=L_{e^{2u}g_0}(\Gamma_i)$ is $C^1$. 

We now claim that $f:(0,1)^k\to (\R^+)^k$ is proper: To see this we first recall that 
Corollary \ref{cor:lower-bound-ell} assures that 
 $L_{g_c}(\Gamma_i)\to \infty$ as $c_i\to 1$ and hence that the preimage $f^{-1}(K)$ of any compact set $K\subset (\R^+_0)^k$ is a compact set in $[0,1)^k$. 
 As $c\mapsto L_{g_c}(\Gamma_i)$ is continuous on $[0,1)^k$ we furthermore have a uniform upper bound on each $L_{g_c}(\Gamma_i)$ for $c\in f^{-1}(K)$. 
For compact subsets $K$ of $(\R^+)^k$ 
 we hence obtain that the components 
$c_i$ of  $c\in f^{-1}(K)$ are bounded away from zero uniformly and hence that $f^{-1}(K)$ is a compact subset of $(0,1)^k$ as required. 
 
By Hadamard's global inverse function theorem, see e.g.~\cite[Chapter 6]{Inverse}, the lemma thus follows provided we show that  
\beqs
\det(df(c))\neq 0 \text{ for every } c\in (0,1)^k.
\eeqs
So suppose that there exists $c\in (0,1)^k$ so that $\det(df(c))=0$. Hence there must be some non-trivial element $b$ of the kernel of $df(c)$, i.e.~$b\in \R^k\setminus\{0\}$ so that for every $i=1,\ldots, k$
\beq 
\label{eq:to-contradict-reg}
0=df(c)(b)_i=b_iL_{g_c}(\Gamma_i)+c_i\int_{\Gamma_i} \tfrac{d}{d\eps}\vert_{\eps=0} e^{u_c+\eps v_{c,b}} \dSgc{0} =b_iL_{g_c}(\Gamma_i)+c_i\int_{\Gamma_i} v_{c,b} \dSgc{c}
 ,\eeq
where $v_{c,b}=dS(c)(b)$ is characterised by \eqref{eq:vbc}. Testing \eqref{eq:vbc} with $v_{c,b}$ and applying the trace estimate \eqref{est:trace-L2} of Lemma \ref{lemma:trace2}  however yields that 
\beqas \int_M\abs{d v_{c,b}}_{g_c}^2+2v_{c,b}^2 \dvgc{c}
=&\int_{\bM} \kgc v_{c,b}^2dS_{g_c} +\sum_i b_i\int_{\Gamma_i} v_{c,b} \dSgc{c}\\
\leq &(1-\eps) \int_M\abs{d v_{c,b}}_{g_c}^2+2v^2 \dvgc{c}+
\sum_i b_i\int_{\Gamma_i} v_{c,b} \dSgc{c}.
\eeqas
Since $v_{c,b}$ cannot vanish identically as $b\neq 0$, there hence must be at least one $i\in\{1,\ldots k\}$ with
$$\text{sign}(b_i) \int_{\Gamma_i}v_{c,b} \dSgc{c}>0$$
which contradicts \eqref{eq:to-contradict-reg} as $c_i>0$.
\end{proof}

Having thus proven that each conformal class is represented by a unique metric $g\in\M^d$, we now obtain the remaining claims of Theorem \ref{thm:1} from the following lemma which is based on Lemma \ref{lemma:collar} and Corollary \ref{cor:lower-bound-ell}.
 
 \begin{lemma} \label{lemma:col-Md}
Let $M$ be as in Theorem \ref{thm:1} and let $g\in \M^d$, $d>0$. Then there is a unique  geodesic $\gamma_i$ in $(M,g)$ homotopic to the boundary curve $\Gamma_i$, its length  
$\ell_i$ is related to the length of $\Gamma_i$ by \eqref{claim:rel-length} and $\Gamma_i$ is surrounded by a collar neighbourhood that is isometric to 
\beq \label{def:collar-Md}
\big((-X(\ell_i),\Xdb(\ell_i)]\times S^1, \rho_{\ell_i}(ds^2+d\th^2)\big) 
\text{ where }\Xdb(\ell)=\frac{2\pi}{\ell}\big(\frac\pi2- \arctan(\frac{\ell}{d})\big)
\eeq
while $X(\ell)$ and $\rho_\ell$ are as in 
 \eqref{def:rho-X}.
\end{lemma} 
  \begin{proof}
  The existence of such a geodesic was proven in Lemma \ref{lemma:collar} and the relation between $\ell_i=L_g(\gamma_i)$ and $L_i=L_g(\Gamma_i)$ follows from Corollary \ref{cor:lower-bound-ell} which implies that 
$\ell_i^2=(1-(k_g\vert_{\Gamma_i})^2)L_i^2=L_i^2-d^2.$
From Lemma \ref{lemma:collar} we then obtain that the boundary curve is surrounded by a collar as described in the above lemma where we know that $\Xdb$ must be so that
$k_g\vert_{\Gamma_i}=\sin(\frac{\ell}{2\pi}\Xdb(\ell_i))$. Combined with 
\eqref{est:lower-bound-ell} this yields the condition 
$\frac{d}{\ell_i}=\frac{k_g\vert_{\Gamma_i} L_i}{\sqrt{1-(k_g\vert_{\Gamma_i})^2}L_i}=\tan(\frac{\ell}{2\pi}  \Xdb(\ell_i))$, so $\Xdb(\ell_i)$ must be given by \eqref{def:collar-Md}.
\end{proof}

\begin{rmk}\label{rmk:L-to-infty}
We remark that while  $X(\ell)$ and $\Xdb(\ell)$ have a similar asymptotic behaviour as  $\ell\to 0$
the behaviour of $X(\ell)$ and $\Xdb(\ell)$ as $\ell\to \infty$ is very different, with 
$X(\ell)$ decaying exponentially, $X(\ell)\leq C\ell^{-1} e^{-\ell/2}$, while $\Xdb(\ell)$ is of order $\ell^{-2}$ for large $\ell$. This difference is significant due to its effect on the  Teichm\"uller space  and its completion with respect to the  corresponding Weil-Petersson metric, which will be discussed in more detail in future work. 
To illustrate this, we note that for the corresponding metrics $G_\ell$ on cylinders (which were considered in \cite{R-cyl} and are isometric to $([-\Xdb(\ell),\Xdb(\ell)]\times S^1,\rho^2(ds^2+d\th^2)$)
the change of the length of the central geodesic is controlled by
$\frac{d\ell}{dt}\leq C\ell \norm{\partial_t G_{\ell(t)}}_{L^2}$ if $\ell$ is large, which excludes the possibility that $\ell\to \infty$ along a curve of metrics of finite $L^2$-length. In this case we thus know that while the completion of the Teichm\"uller space includes the punctured limit obtained as $\ell\to 0$, it does not include any limiting object corresponding to $\ell$ becoming unbounded.
\end{rmk}

\begin{rmk}\label{rmk:horizontal}
We note that the quantity $L_g(\Gamma_i)^2-L_g(\gamma_i)^2$ appears naturally also for horizontal curves of hyperbolic metrics on closed surfaces. Such curves move orthogonally to the action of diffeomorphisms and hence satisfy $\pt \hat g =\Rea(\Om)$ for 
holomorphic quadratic differentials $\Om$ on $(M,\hat g)$.
Given a simple closed geodesic $\gamma\subset (M,\hat g)$ and a 
  closed curve $\Gamma\subset (\Col(\gamma),\hat g)$ with  constant geodesic curvature, which is hence described  by some $\{s\}\times S^1$ in collar coordinates, we can use the  Fourier expansion of $\Om=\sum_{j\in\Z} b_j e^{j(s+\i\th)}dz^2$ to obtain that 
\beqas
 \frac{d}{dt} \big[ L_{\hat g(t)}(\Gamma)^2-L_{\hat g(t)} (\gamma)^2\big]
&=  L_{\hat g}(\Gamma)\int_{\{s\}\times S^1} \frac{(\pt \hat g)_{\th\th}}{\sqrt{\hat g_{\th\th}}} d\th
-L_{\hat g}(\gamma)\int_{\{s\}\times S^1} \frac{(\pt \hat g)_{\th\th}}{\sqrt{\hat g_{\th\th}}} d\th \\
&=- 2\pi\Rea\big(\sum_{j} b_j (e^{js}-1) \int_{S^1} e^{\i j \th} d\th\big)=0.
\eeqas
\end{rmk}

We finally give the proof of the analogue of the Deligne-Mumford compactness result for our class of metrics that we stated in Theorem \ref{thm:DM}. This proof is based on the proof of the corresponding 
result  for surfaces with geodesic boundary curves as carried out in \cite[Sect. IV.5]{Hu}.
For part of this proof it will be more convenient to work with so called Fermi coordinates $(x,\th)$ instead of collar coordinates $(s,\th)$ on a collar $\Col(\si)$ around a simple closed geodesic $\si$. As indicated the angular components of these two different sets of coordinates agree, while the $x$ coordinate of a point $p\in\Col(\si)$ is given as the signed distance $x=x(s)=\text{dist}_g(p,\si)$ 
to $\si$.  

\begin{proof}[Proof of Theorem \ref{thm:DM}]
For $(M,g^{(j)})$ as in the theorem we denote by $\gamma^{(j)}_i$, $i=1,\ldots,k$,  the (unique) geodesics 
in  $(M,g^{(j)})$ 
that are  homotopic to $\Gamma_i$, and note that their lengths are bounded from above by $\ell_i^{(j)}:=L_{g^{(j)}}(\gamma_i^{(j)})\leq L_{g^{(j)}}(\Gamma_i)\leq C$. We can thus pass to  subsequence so that $\ell_i^{(j)}\to \ell_i^\infty$ as $j\to \infty$ for each $i=1,\ldots, k$, where, after relabelling,  we may assume that 
$\ell_i^\infty=0$ for $1\leq i\leq \kappa_1$ while $\ell_i>0$ for $\kappa_1+1\leq i\leq k 
$ for some $\kappa_1\in\{0,\ldots, k\}$.
As above we let $\Col^+(\Gamma_i,g^{(j)})$, $i=1,\ldots,k$, be the halfcollars that are bounded by $\Gamma_i$ and the corresponding geodesic $\gamma^{(j)}_i\subset (M,g^{(j)})$ and set $ M^{(j)}:=M\setminus \bigcup_i \Col^+(\Gamma^{(j)}_i)$.

As $(M^{(j)},g^{(j)})$ is a sequence of hyperbolic surfaces with geodesic boundary we can apply the version of the Deligne-Mumford compactness theorem as found in \cite[Prop.~5.1]{Hu}, or alternatively first double the surface and then apply the version of Deligne-Mumford for closed surfaces that is recalled e.g.~in \cite[Prop. A.3]{RT-neg}. After passing to a subsequence we thus obtain 
a collection 
 $\mathscr{E}^{(j)}=\{\sigma_i^{(j)}, i=1,\ldots,\kappa_2\}$, $\ka_2\in\{0,\ldots,3(\gamma-1)+k\}$, of simple closed geodesics in the interior of $(M^{(j)},g^{(j)})$ whose lengths tend to zero, so that for the 
  surfaces 
  $\Si^{(j)}= M^{(j)}\setminus \big(\bigcup_{i=1}^{\kappa_2} \sigma_i^{(j)}\cup\bigcup_{i=1}^{\kappa_1} \gamma_i^{(j)}\big)$, which have with $\kappa_1+2\kappa_2$ punctures, the following holds true: There exist a complete hyperbolic
metric $\hat g_\infty$ on $\hat \Si:=\Si^{(1)}$ and  diffeomorphisms 
$\hat f_j: \hat \Si\to \Si^{(j)}$
which map
$\gamma_i^{(1)}$ to $\gamma_i^{(j)}$, $i\geq \kappa_1+1$, as well as neighbourhoods of each 
puncture (respectively of each pair of punctures) of $(\Si^{(1)},g_\infty)$ obtained by collapsing one the $\gamma_i^{(1)}$, $1\leq j\leq \kappa_1$ (respectively one of the $\si_i^{(1)}$) to a neighbourhood of the corresponding puncture (respectively pair of punctures) on $\Si^{(j)}$, 
so that
$$\hat f_j^{*}g^{(j)}\to \hat g_\infty \text{ smoothly locally on }\hat \Si= \Si^{(1)}.$$
Furthermore, for $j$ sufficiently large we can modify these diffeomorphisms 
as described in the proof of Claim 3 on p.~75 of \cite{Hu} to ensure that 
$\hat f_j: (\hat \Si,g_\infty) \to (\Si^{(j)},g^{(j)})$ is given in a neighbourhood of  $\bigcup_{i=\kappa_1+1}^k\gamma^{(1)}_i$ by the identity in the respective Fermi coordinates.

In slight abuse of notation we now denote by $\Col^+(\ell)$, $\ell\geq 0$, the unique hyperbolic half-collar which has one boundary curve of constant geodesic curvature and length $L$, where $L^2-\ell^2=d^2$, while the other boundary curve is a geodesic of length $\ell$ if $\ell>0$, respectively degenerated to a hyperbolic cusp if $\ell=0$. 
We then construct the limit surface $(\Si,g_\infty)$ out of the limiting surface $(\hat \Si,\hat g_\infty)$ with geodesic boundary obtained above and the half collars $\Col^+(\ell_i^\infty)$,  
$i=1,\ldots ,k$, by gluing the nondegenerate half-collars  $\Col^+(\ell_i^\infty)$, $i\geq \kappa_1+1$, to $\hat\Si$ along the corresponding non-collapsed boundary curves of $(\hat \Si,\hat g_\infty)$, and adding the degenerate collars  $\Col^+(\ell_i^\infty)$, $i\leq \kappa_1$, as additional connected components of $(\Si,g_\infty)$. 
As the connected components of  $M\setminus \bigcup_{i=1}^{\kappa_1}\gamma^{(j)}_{i}$ are given by $M\setminus \bigcup_{i=1}^{\kappa_1} \Col^+(\ell_i^{(j)})$ and 
a collection $\{\Col^+(\ell_i^{(j)})\}_{i=1}^{\kappa_1}$  of degenerating halfcollars, we can now extend  
the diffeomorphisms $\hat f_j$ obtained above to the required diffeomorphisms
$$f_j:\Si\to  M\setminus \bigg(\bigcup_{i=1}^{\kappa_1}\gamma^{(j)}_{i}\cup \bigcup_{i=1}^{\kappa_2}\si^{(j)}_{i}\bigg)$$
as follows: 
The degenerated 
connected components $\Col^+(\ell_i^{\infty}=0)$,
$1\leq i\leq \kappa_1$, 
 which are isometric to 
$([0,\infty)\times S^1, \rho_{\ell=0}^2(ds^2+d\th^2))$, $\rho_0(s)=\frac{1}{d+s}$, are mapped to the degenerating half-collars $\Col^+(\ell_i^{(j)})$ in $(M,g^{(j)})$ which are bounded by $\gamma_i^{(j)}$ and $\Gamma_i^{(j)}$ with $f_j$ chosen so
that it is given in collar coordinates by a bijection from $[0,\infty)\times S^1$ to $(0,\Xdb(\ell_i^{(j)})]\times S^1$ 
with $f_j(s,\th)=(\Xdb(\ell_i^{(j)})-s,\th)$ on domains  which exhaust $[0,\infty)\times S^1$, say for $s\in [0,\half \Xdb(\ell_i^{(j)})]$. As $\rho_{\ell}(\bar X_{d}(\ell_i^{(j)})-\cdot)\to \rho_{0}(\cdot) $ locally uniformly on $[0,\infty)$ as $\ell\to 0$, this ensures that the pulled back metrics converge on every compact subset of these connected components of the limit surface as required.

Finally we extend $f_j$ to the collars that we glued to $\hat\Si$ as follows: We let
 $w_i^{+,\infty}$ and $w_i^{+,(j)}$, $i\geq \kappa_2+1$, be the width of the half-collars  $\Col^+(\ell_i^{\infty})$ and  $ \Col^+(\ell_i^{(j)})$, i.e.~the geodesic distance between the two boundary curves, and note that $w_i^{+,(j)}\to w_i^{+,\infty}$ since $\ell_i^{(j)}\to \ell_i^\infty$. We may thus choose smooth bijections
$\phi_i^{(j)}:[0,w_i^\infty]\to [0,w_i^{(j)}]$ which agree with the identity in a neighbourhood of $0$ and converge to the identity as $j\to \infty$. Since  $\hat f_j:\hat\Si\to \hat\Si^{(j)}$ is given by the identity in Fermi-coordinates near the boundary curves, we may extend it to a smooth diffeomorphism on $\Si$ for which 
 $f_j^*g^{(j)}$ converges as claimed in Theorem \ref{thm:DM}
by defining $f_j$ on 
$\Col^+(\ell_i^{\infty})$ by  $f_j(x,\theta)=(\phi_i^{(j)}(x),\theta)$ in Fermi coordinates.
\end{proof}
\section{Appendix}
We finally include a proof of Remark \ref{rmk:cylinder-Z}, which uses that the Dirichlet energy $E(u)=\half \int \abs{du}_g^2 dv_g$ is conformally invariant and hence well defined for functions on Riemann surfaces.
\begin{proof}[Proof of Remark \ref{rmk:cylinder-Z}]
If $M$ has at least two boundary curves, we use that
the harmonic function $\bar f$ which is $1$ on $\Gamma_i$ and zero on all other boundary components, minimimises $E$ among all functions with the same boundary data. 
Given a cylindrical neighbourhood of $\Gamma_i$ as in the remark, we thus have
$E(f)= \frac{\pi}{Z}\geq \de:=E(\bar f)$
for the function $f$ which is linear on $U\sim (0,Z]\times S^1$ and zero elsewhere, and thus $Z\leq \frac{\pi}{\de}$.

If $M$ has only one boundary curve, we fix instead some   curve $\si$ with endpoints on the boundary curve $\Gamma$ so that $\si$ is
homotopically non-trivial with respect to variations by curves with endpoints on $\Gamma$. Then there exists a simple closed curve $\gamma$ in the interior of $M$ so that any curve $\si'$ which is homotopic to $\si$ (with endpoints on $\Gamma$) must intersect any curve $\gamma'$ that is homotopic to $\gamma$. 
We claim that there exists some $\de>0$ so that $E(f)\geq \de$ for all functions $f:M\to \R$ which are equal to $1$ on $\Gamma$ and for which there is a curve $\gamma'$ homotopic to $\gamma$ so that $f\vert_{\gamma'}\leq 0$. As in the first case, this will then imply that $Z\leq \frac{\pi}{\de}$.
To prove the claim, we fix a neighbourhood $V$ of the fixed curve $\si$ that is diffeomorphic (but not necessarily conformal) to some rectangle $R=[-c,c]\times [-b,b]$, say with $\si$ corresponding to $\{0\}\times [-b,b]$ and with $V\cap \Gamma$ corresponding to $[-c,c]\times \{-b\}\cup [-c,c]\times \{b\}$ for the chosen diffeomorphism $\phi :R\to V$ and fix some smooth metric $g$ on $M$ that is compatible to $\cc$. 
Using that $\phi^*g$ is equivalent to the euclidean metric on $R$ we obtain that there exists $c_0>0$ (allowed to depend on the above construction) so that for any $f:M\to \R$ as considered above and $\tilde f:=f\circ \phi$ 
\beqas
E(f)&
\geq \half
\int_R \abs{d\tilde f}_{\phi^*g}^2 dv_{\phi^*g}\geq c_0 \int_R \abs{\partial_x \tilde f}^2+\abs{\partial_y \tilde f}^2 dx\,dy \geq c_0 cb^{-1} \inf_{a\in [-c,c]} \big(\int_{\{a\}\times [-b,b]} \abs{\partial_y\tilde f} dy\big)^2.
\eeqas 
As the curves $\phi(\{a\}\times [-b,b])$, $a\in[-c,c]$ are homotopic to $\si$ and thus intersects the curve $\gamma'$ for which $f\vert_{\gamma'}\leq 0$,  while $\tilde f(a, \pm b)=1$, we thus get $E(f)\geq \de:= 4c_0cb^{-1}>0$ as claimed.
\end{proof}

{\sc 
Mathematical Institute, University of Oxford, Oxford, OX2 6GG, UK


\begin{thebibliography}{99}
\parskip 1pt
\setlength{\itemsep}{1pt plus 0.3ex}
\bibitem{Berger} M.S. Berger,  \emph{Riemannian structures of prescribed Gaussian curvature for compact 2-manifolds}, J. Differential Geom.  \textbf{5} (1971), 325--332.

\bibitem{Borer-G-Struwe}
F. Borer, L. Galimberti, M. Struwe, \emph{ "Large" conformal metrics of prescribed Gauss curvature on surfaces of higher genus}, Comment. Math. Helv. \textbf{90} (2015), 407--428.


\bibitem{Brendle-1} S. Brendle, \emph{Curvature flows on surfaces with boundary}, Math. Ann. \textbf{324} (2002), 491--519. 


\bibitem{Brendle-2} S. Brendle, \emph{A family of curvature flows on surfaces with boundary}, Math. Z. \textbf{241} (2002), 829--869. 

\bibitem{Reto-Melanie} R. Buzano and M. Rupflin, \emph{Smooth long-time existence of harmonic Ricci flow on surfaces}, J. Lond. Math. Soc. \textbf{95} (2017),  277--304.


\bibitem{Chang-Yang} S.Y.A. Chang, P. Yang, \emph{Conformal deformation of metrics on $S^2$}, 
J. Differential Geom. \textbf{27} (1988), 259--296. 

\bibitem{Chen} W. Chen and C. Li, \emph{Gaussian curvature in the negative case}, Proc. Amer. Math. Soc. \textbf{131} (2003) 741--744.

\bibitem{Cherrier} P. Cherrier, \emph{ Probl\`emes de Neumann non lin\'eaires sur les vari\'eet\'es riemanniennes}, Journal of Functional Analysis, \textbf{57} (1984), 154--206.


\bibitem{Ding-Li-Liu} W. Ding, J. Li and Q. Liu, \emph{Evolution of minimal torus in Riemannian manifolds}, Invent. Math. \textbf{165} (2006) 225--242.
 
 \bibitem{Eells-Sampson} J. Eells and J. H. Sampson, \emph{Harmonic mappings of Riemannian manifolds}, Amer. J. Math. \textbf{86} (1964), 109--160.
%

 \bibitem{Hu} C. Hummel: \emph{Gromov's compactness theorem for pseudo-holomorphic curves}, Progress in Mathematics, \textbf{151},
 Birkh\"{a}user Verlag, Basel, (1997), viii+131 pp.
%

\bibitem{Inverse}  S. Krantz  and H. Parks \emph{The Implicit Function Theorem, History, Theory and Applications},  Birkh\"auser Boston, Inc., Boston, MA, (2002).
%
\bibitem{Kazdan-Warner} 
J. Kazdan and F. W. Warner, \emph{ Curvature functions for compact 2-manifolds}, Ann. of Math. \textbf{99} (1974), 14--47. 


\bibitem{Li-Liu-MT} 
X. Li, P. Liu, \emph{ A Moser-Trudinger inequality on the boundary of a compact Riemann surface}, Math. Z. \textbf{250} (2005), 363--386. 


\bibitem{Mazzeo-Taylor} R. Mazzeo, M. Taylor, Michael \emph{Curvature and uniformization}, Israel J. Math. \textbf{130} (2002), 323--346. 

\bibitem{Reto} R. M\"uller, \emph{Ricci flow coupled with harmonic map flow},
Ann. Sci. Éc. Norm. Supér.  \textbf{45} (2012), 101--142

\bibitem{Sarnak}
B. Osgood, R. Phillips, P. Sarnak, \emph{Extremals of determinants of Laplacians}, Journal of Functional Analysis
\textbf{80} (1988), 148--211.

 \bibitem{randol} B. Randol, \emph{Cylinders in Riemann surfaces},
 Comment. Math. Helvetici \textbf{54} (1979) 1--5.


 \bibitem{RT} M. Rupflin and P. M. Topping, \textit{Flowing maps to minimal surfaces},  Amer. J. Math. \textbf{138} (2016),  1095--1115.

 \bibitem{RT-neg} M. Rupflin and P.M. Topping, 
 \emph{Teichmu\"uller harmonic map flow into nonpositively curved targets}, J. Differential Geom. \textbf{108} (2018), 135--184.
 
\bibitem{RT-global} M. Rupflin and P.M. Topping, \emph{Global weak solutions of the Teichm\"uller harmonic map flow into general targets}, To appear in Analysis and PDE, preprint (2017) \url{http://arXiv:1709.01881}
 
 
\bibitem{R-cyl} M. Rupflin, \emph{ Teichm\"uller harmonic map flow from cylinders},  Math. Ann. \textbf{368} (2017), no. 3-4, 1227--1276.

\bibitem{Schoen-Yau} R. Schoen, S.T. Yau \emph{Lectures on differential geometry}, International Press, Cambridge, MA, (1994). 

\bibitem{Struwe-flow}
M. Struwe \emph{A flow approach to Nirenberg's problem}, Duke Math. J. \textbf{128} (2005), no. 1, 19--64.


\bibitem{Moser-Trudinger} N. Trudinger, \emph{On imbeddings into Orlicz spaces and some applications}, J. Math. Mech. \textbf{17} (1967) 473--483.
\end{thebibliography}
\end{document}